\documentclass[11pt,a4paper]{amsart}
\numberwithin{equation}{section}
\usepackage{rotating}
\usepackage{mathrsfs}
\usepackage{amscd,amssymb,amsopn,amsmath,amsthm,graphics,amsfonts,enumerate,verbatim,calc}
\usepackage[all]{xy}
\usepackage[utf8]{inputenc}
\usepackage{color}

\usepackage[a4paper,margin=1in]{geometry}
\usepackage[utf8]{inputenc}
\usepackage{amscd,amssymb,amsopn,amsmath,amsthm,graphics,amsfonts,enumerate,verbatim,calc}
\usepackage[T1]{fontenc}
\usepackage{lmodern}
\usepackage{enumitem}
\usepackage{hyperref}
\usepackage{microtype}\usepackage{hyperref}
\usepackage{geometry}
\usepackage{enumitem}
\usepackage{booktabs}
\usepackage{tikz-cd}
\hypersetup{colorlinks=true,linkcolor=blue,citecolor=blue,urlcolor=blue}

\newtheorem{theorem}{Theorem}[section]
\newtheorem{proposition}[theorem]{Proposition}
\newtheorem{corollary}[theorem]{Corollary}
\newtheorem{lemma}[theorem]{Lemma}
\newtheorem{fact}[theorem]{Fact}
\newtheorem{example}[theorem]{Example}
\newtheorem{observation}[theorem]{Observation}
\newtheorem{remark}[theorem]{Remark}
\newtheorem{question}[theorem]{Question}

\newtheorem{definition}[theorem]{Definition}
\newcommand{\Min}{\operatorname{Min}}
\newcommand{\rad}{\operatorname{rad}}
\newcommand{\Spec}{\operatorname{Spec}}
\newcommand{\Ext}{\operatorname{Ext}}
\newcommand{\Hom}{\operatorname{Hom}}
\newcommand{\depth}{\operatorname{depth}}
\newcommand{\Ann}{\operatorname{Ann}}

\newcommand{\Ht}{\operatorname{ht}}

\newcommand{\pd}{\operatorname{pd}}

\newcommand{\fm}{\frak{m}}
\DeclareMathOperator{\coker}{coker}

\begin{document}
	
\author[M. Asgharzadeh and E. Mahdavi]{Mohsen Asgharzadeh and Elham Mahdavi}
\date{}

\address{M. Asgharzadeh}
\email{mohsenasgharzadeh@gmail.com}

\address{E. Mahdavi}
\email{elham.mahdavi.gh@gmail.com}

	\title[Embedding  into   Quotient  ]
{Embedding  into  Quotient}

\subjclass[2020]{13A15, 13E15.}
	\keywords{Cohen- Macaulay; embeddings and quotients; finiteness properties; reduced    rings; nilpotent.  }

	\begin{abstract}
We study Bass's finiteness problem for embeddings into proper quotients. Given a
commutative ring \(R\), a reduced finitely generated \(R\)-algebra \(A\), and an
ideal \(J\triangleleft A\), we ask whether an \(R\)-algebra monomorphism
\(A\hookrightarrow_R A/J\) forces \(J=0\). We give affirmative answers by presenting two arguments. Also, we discuss not only in the reduced case, but  also in several
non-reduced cases:  when \(R\) is artinian and \(A\) is finite over \(R\); when \(R\) is
zero-dimensional and \(A\) is finite projective over \(R\); and in a series of
Cohen--Macaulay settings. In the module-theoretic direction,
for a noetherian reduced ring, and a
finitely generated, and torsion-free \(A\)-module \(M\), we show that
\(M\hookrightarrow_R M/JM\) implies that \(J\) is contained in a minimal prime of
\(A\); if \(M\) is faithful, then \(J=0\). Also, dual versions investigated.
 Examples illustrate the necessity of  hypotheses.
\end{abstract}

	\maketitle
	 \tableofcontents

\section{Introduction}
	 
This paper starts with the following implicit question of Bass \cite{Bass}.

\begin{question}\label{Bb}
	If \(A\) is a finitely generated reduced commutative \(R\)-algebra and
	\(A\hookrightarrow A/J\) is an \(R\)-algebra monomorphism, must \(J=0\)?
\end{question}

Bass proved:

\begin{theorem}\label{bt}(Bass, see \cite[Main result]{Bass}).
	Let $R$ be a commutative ring, let $A$ be a reduced finitely presented
	$R$-algebra, and let $J$ be an ideal of $A$. If there is an injective
	$R$-algebra monorphism $A\to A/J$, then $J=0$.
\end{theorem}

Also, he showed that the finitely generated hypothesis is necessary. Indeed, for \(A=R[X_1,X_2,\ldots]\), there is an \(R\)-algebra isomorphism \(A\cong_R A/(X_1)\), but \(A\) is not finitely generated. He also showed the importance of the reduced assumption (see \cite[Example 1.3]{Bass}).

The purpose of the present paper is to study Question \ref{Bb} and several variants, both in the algebra and module settings. We begin with the following result (later we improve it by replacing domain with reduced):

\begin{theorem}\label{ia} (See Theorem \ref{a}).
	Let \(R\) be a domain, let \(A\) be a reduced finitely generated \(R\)-algebra which is torsion-free as an \(R\)-module, and let \(J\) be an ideal of \(A\). If there exists an \(R\)-algebra monomorphism
$
	\varphi:A\hookrightarrow_R A/J,
$
	then
$
	J=0.
	$
\end{theorem}

In Section 2 we also treat non-reduced algebras and prove an affirmative answer in several cases. We then relate Question \ref{Bb} to (co-)Hopficity. Recall that an \(R\)-algebra \(A\) is Hopfian if every surjective \(R\)-endomorphism \(A\to A\) is injective. If \(A\cong_R A/J\) for some nonzero ideal \(J\), then \(A\) admits a surjective non-injective \(R\)-endomorphism; moreover, \(A\ncong_A A/J\). We also prove several further positive results. The zero-dimensional case is considered in Lemma \ref{lem:zero-dimensional-base} and Theorem \ref{thm:artinian-base}. In the Cohen--Macaulay direction, Proposition \ref{thm:CM-reduction} shows that, under the natural maximal Cohen--Macaulay hypotheses, an \(R\)-algebra monomorphism \(A\hookrightarrow A/J\) again forces \(J=0\). The proof uses induction on \(\dim R\), reduction modulo a maximal regular sequence, and Krull's intersection theorem. In the opposite direction, Proposition \ref{op} presents a vanishing result for \(\operatorname{coker}(\varphi:A\hookrightarrow_R A/J)\) under a constructive hypothesis starting with a non-Cohen--Macaulay ring \(R\); this again forces \(J=0\).

In Section 3 we show:

\begin{theorem}\label{ia2} (See Theorem \ref{thm:answer}).
Question \ref{Bb} is true.
\end{theorem}

Section 4 develops a module-theoretic analogue. For a noetherian reduced ring \(R\), an ideal \(J\triangleleft R\), and a finitely generated torsion-free \(R\)-module \(M\), Theorem \ref{31} proves that
\[
M\hookrightarrow_R M/JM
\quad\Longrightarrow\quad
J\subseteq\mathfrak p
\text{ for some }\mathfrak p\in\operatorname{Min}(R).
\]
If \(M\) is additionally faithful over \(R\), then \(J=0\). We record a connection to local cohomology in Proposition \ref{lc}. We also discuss the corresponding result for divisible modules; see Proposition \ref{inj}. A non-noetherian example shows that the noetherian hypothesis cannot be dropped.

\begin{theorem}\label{i41}(See {Theorem} \ref{41}).
	Let \(R\subseteq A\) be an extension of noetherian reduced rings such that \(A\) is torsion-free as an \(R\)-module. Let \(J\triangleleft A\), and let \(M\) be a finitely generated \(R\)-module which is torsion-free as an \(A\)-module. If
$
	M\hookrightarrow_R M/JM,
$
	then
$
	J\subseteq\mathfrak p
	$
	for some \(\mathfrak p\in\operatorname{Min}(A)\).
\end{theorem}

We also record a rank version and give an example showing that the finiteness of \(M\) as an \(R\)-module is essential.

Finally, Section 5 is devoted to the dual of Question \ref{Bb}: when does the existence of an \(R\)-linear epimorphism \(A\twoheadrightarrow J\) or \(J\twoheadrightarrow A\) force \(J\cong_R A\)? The answer depends heavily on the hypotheses: over a domain \(A\) finite over \(R\), any such epimorphism is automatically an isomorphism (see Theorem \ref{thm:dual-domain}); in the artinian case, an epimorphism \(J\twoheadrightarrow A\) is necessarily an isomorphism (see Proposition \ref{prop:artinian-dual}), while an epimorphism \(A\twoheadrightarrow J\) need not be. In the Cohen--Macaulay setting, we show that an epimorphism \(J\twoheadrightarrow A\) is impossible under natural conditions (see Proposition \ref{mcmno}). As an application, we show:

\begin{corollary}(See {Corollary} \ref{ccormcmno}).
	Suppose \(R=A\) is a local ring of dimension one and \(P\in \operatorname{Ass}(A)\). Then there is no \(\pi:P\twoheadrightarrow_R A\).
\end{corollary}

We present several examples to show that the assumptions are needed and that the results are sharp.

All rings in the sequel are commutative. We consider both noetherian and non-noetherian rings. The term quasilocal denotes a ring with a unique maximal ideal. By local ring we mean a noetherian quasilocal ring. We follow the notation and conventions of \cite{BH,Matsumura} for homological and commutative algebra.

	\maketitle

\section{Embeddings into proper quotients}

Recall from {Question} \ref{Bb} we are inserted  to \(R\)-algebra embedding
\(
A\hookrightarrow A/J.
\)
The phenomenon is closely related to the notion of Hopficity.
Recall that
an $R$-algebra   $A$ is called \emph{Hopfian} if every surjective
	$R$-endomorphism
	\(
	A\to A
	\)
	is injective.

\begin{proposition}\label{prop:hopfian}
	Let \(A\) be an \(R\)-algebra and let \(0\neq J\subseteq A\) be an ideal.
	If
 	$ 	A\cong_R A/J,
	$ 	then \(A\) admits a surjective \(R\)-algebra endomorphism which is not
	injective. Moreover,
 $
	A\not\cong_A A/J
	$ 	as \(A\)-modules.
\end{proposition}

\begin{proof}
	Let
 $
	\theta:A/J\overset{\cong}{\longrightarrow}_R A
	$ 	be an \(R\)-algebra isomorphism, and let
 $
	q:A\twoheadrightarrow A/J
	$ 	be the canonical quotient map. Then
 $
	\theta\circ q:A\to A
	$ 	is a surjective \(R\)-algebra endomorphism with
 $
	\ker(\theta\circ q)=J\neq0.
	$ 	Thus it is not injective.
 	Finally, suppose that \(A\cong_A A/J\). Since \(J\) annihilates \(A/J\),
	\(A\)-linearity would imply that \(J\) annihilates \(A\). Hence
$
	J=J\cdot1_A=0,
	$	a contradiction.
\end{proof}

We record the following easy observation.

\begin{observation}\label{ffob}
Suppose there is $A'$ between $R\subseteq A$ so that  for some $J'\subseteq A'$
we have $J'A=J$ and that $A'\subseteq A$ is faithfully flat. If   {Question} \ref{Bb}
is valid for $(R,J'\subseteq A')$ then it is valid over $(R,J\subseteq A)$.
\end{observation}

\begin{corollary}
		Let \(A'\) be a finitely related reduced \(R\)-algebra and $J'$ be an ideal.
		Set $A=A'[Z]/(Z^n)$ and $J=J'A$.
If	\(
	A\hookrightarrow A/J
	\)
	is an \(R\)-algebra monomorphism, then \(J=0\).
\end{corollary}

\begin{proof}Recall that $A$ is free over $A'$, because whose base is given by $\{1,Z,\ldots, Z^{n-1}\}.$ It is easy to see that	\(
	A\hookrightarrow A/J
=(A'\to A'/J')\otimes_{A'}A	\).
By faithfully flatness, $A'\to A'/J'$ is injective. By \cite[Main result]{Bass}, $J'=0$. Then by 	
{Observation} \ref{ffob} we get the claim.	
\end{proof}
Before presenting the next result, we present a useful example.

\begin{example}(Warm-up)
	Let $R$ be a domain and set
$
	A=R[X,Y,Z]/(XY,Z^2).
	$ 	Write $x$ and $y$ for the images of $X$ and $Y$ in $A$. 
Let $J=(x-y)$. Then   there is no
	injective $R$-algebra homomorphism
$
	\varphi:A\hookrightarrow_R A/J.
	$\end{example}

\begin{proof} Without loss of generality, we may replace $A$ with
	$
R[X,Y]/(XY)
	$ (see {Observation} \ref{ffob}). We are going to prove the claim without any use of {Theorem} \ref{bt}.
	 Suppose not. So, there is such a morphism $
	\varphi$.
	First observe that, as an $R$-module,
 $
	A=R\oplus xR[x]\oplus yR[y] 
	$	
 with $
	(x)\cap(y)=0.
	$ 	 Then
 $
	\operatorname{Ann}_A(x)=(y),
	\operatorname{Ann}_A(y)=(x).
	$ 	Put
$
	u=\varphi(x),  v=\varphi(y).
$ Since $xy=0$, we have
$
	uv=\varphi(xy)=0.
	$ Since $\varphi$ is injective and $x,y\neq0$, we have
 $
	u\neq0\neq v.
$
Then
$
	A/(x-y)\cong R[x]/(x^2).
	$ Every element of $R[x]/(x^2)$ has a unique expression
$
	a+bx \text { with } a,b\in R.
	$ 	Suppose
$
	(a+bx)(c+dx)=0.
	$ 	Since $x^2=0$, we obtain
$
	ac+(ad+bc)x=0.
$
	Hence
$ac=0$ and $ad+bc=0.$ 	Since $R$ is a domain, $ac=0$ implies
 $	a=0 \text{ or } c=0.$
 If $a=0$, then
 $bc=0.$ 	Thus either $b=0$, in which case $a+bx=0$, or $c=0$.
	Similarly, if $c=0$, then
 $ad=0,$ 	so either $a=0$ or $d=0$.
 Consequently, if both $a+bx$ and $c+dx$ are nonzero, then necessarily
 $	a=c=0.$ 	Thus every pair of nonzero elements with zero product has the form
$bx, dx,$ with $b,d\neq0$. In particular,
 $(bx)^2=(dx)^2=0.	$ Now let
 $u=\varphi(x), v=\varphi(y).$ 	We have $u,v\neq0$ and $uv=0$, so the preceding observation gives
 $	u^2=v^2=0.$ But
 $u^2=\varphi(x^2)$ and $v^2=\varphi(y^2).$ Since
 $x^2\neq0\neq y^2 $ in $A$, this contradicts the injectivity of $\varphi$. Hence
$A\not\hookrightarrow_R A/(x-y).$
\end{proof}

\begin{theorem}
	\label{a}
	Let \(R\) be a  domain, let \(A\) be a reduced finitely generated
	\(R\)-algebra which is torsion-free as an \(R\)-module, and let \(J\) be an
	ideal of \(A\). If there exists an \(R\)-algebra monomorphism
 $
	\varphi:A\hookrightarrow_R A/J,
 $
 	then
	$
	J=0.
	$ \end{theorem}

\begin{proof}
	Let \(K=\operatorname{Frac}(R)\). Since \(A\) is torsion-free over
	the domain \(R\), the canonical map
	
	$$
	A\longrightarrow K\otimes_R A=:A_K
	$$
	
	is injective. Hence \(J=0\) if and only if \(J_K:=K\otimes_R J=0\).
	Since \(K\) is flat over \(R\), tensoring \(\varphi\) with \(K\) gives
	a \(K\)-algebra monomorphism
	
	$$
	\varphi_K:A_K\longrightarrow A_K/J_K.
	$$
	
	Moreover, \(A_K\) is a finitely generated reduced \(K\)-algebra.
	Indeed, \(A_K\) is obtained from \(A\) by base change to the field
	\(K\), and reducedness is preserved under field extensions.
	Since \(K\) is a field, every finitely generated \(K\)-algebra is
	finitely presented. Thus \(A_K\) is a finitely presented reduced
	\(K\)-algebra. Bass' theorem (see {Theorem} \ref{bt})
	therefore applies to
$\varphi_K:A_K\to A_K/J_K,
	$	
	and yields
$J_K=0.	$	Finally, \(J\) is an \(R\)-submodule of the torsion-free \(R\)-module
	\(A\), and hence \(J\) is torsion-free. Therefore
$K\otimes_RJ=0
	$	implies \(J=0\).
\end{proof}

\begin{corollary}\label{cor:domain}
	Let \(R\) be a  ring and let \(A\) be a finitely generated
	\(R\)-algebra which is a domain. If \(J\subseteq A\) is an ideal and there
	exists an \(R\)-algebra monomorphism
 $
	A\to_R A/J,
	$
	then
 $J=0.
	$ 	
\end{corollary}

\begin{proof}
	Let
 $
	\rho:R\to A
	$ 	be the structure homomorphism and put
 $
	\mathfrak p=\ker(\rho).
	$ 	Since \(A\) is a domain, \(\mathfrak p\) is prime. Set
 $
	R_0=R/\mathfrak p.
	$ 	Then \(R_0\) is a domain and the structure map factors as
 $
	R\to R_0\hookrightarrow A.
	$ 	Since the given map is an \(R\)-algebra homomorphism, it is also an
	\(R_0\)-algebra homomorphism. Moreover, \(A\) is a finitely generated
	\(R_0\)-algebra and is a domain. Hence Theorem~\ref{a}
	applies and gives \(J=0\).
\end{proof}

Artianian rings are far from be reduced.  This suggests the next result:
	
	\begin{lemma}
		\label{thm:artinian-base}
		Let $R$ be an artinian  ring and let $A$ be a finite
		$R$-algebra. If $J$ is an ideal of $A$ and there exists an
		$R$-algebra monomorphism
		\(
		\varphi:A\hookrightarrow_RA/J,
		\)
		then
		\(
		J=0.
		\)
	In particular, the conclusion holds without assuming that $A$ is
		reduced.
\end{lemma}
	
\begin{proof} Since $R$ is artinian and $A$ is finite as an $R$-module, $A$ is an
Artinian $R$-module. In particular, $A$ has finite length as an
$R$-module. The quotient $A/J$ also has finite length.
Since $\varphi$ is an $R$-algebra monomorphism, it is in particular an
injective homomorphism of $R$-modules. Hence
\(	\ell_R(A)\leq\ell_R(A/J),	\) with the convenience that $\ell(-)$ means the length function.
	On the other hand, the canonical exact sequence
\[0\longrightarrow J\longrightarrow A\longrightarrow A/J
\longrightarrow0
\]
	gives	\(\ell_R(A)	=\ell_R(J)+\ell_R(A/J).\)
Consequently,
\(\ell_R(A/J)\leq\ell_R(A).\)
Combining this  with the previous  inequality  
	\(	\ell_R(A)\leq\ell_R(A/J)\) yields
\(\ell_R(A)=\ell_R(A/J),
\) and so
\(	\ell_R(J)=0.\)
Therefore
\(	J=0.\)
\end{proof}

Here, we relate the problem to co-Hopfianity,  indeed, this  is a natural extension from artinian rings to zero-dimensional rings (not necessarily noetherian):
\begin{lemma}
	\label{lem:zero-dimensional-base}
	Let $R$ be a  ring such that every prime ideal of $R$
	is maximal. Let $A$ be a finite $R$-algebra. Suppose $J$ is an ideal of
	$A$ and there exists an $R$-algebra monomorphism
$A\hookrightarrow_R A/J.
$
If $J=(0:a)	$
	then
$J=0.
	$
	\end{lemma}

\begin{proof}
First recall that  $A$ is finite as an $R$-module, and recall that
$$
A\hookrightarrow_R A/J=A/(0:_Aa)=aA\subseteq A,
	$$
	is invective.
By  the main result of \cite{V}	this map should be surjective. This implies that $aA=A$, i.e., $a$ is invariable, 
	we conclude from this that $J=(0:a)=0$.
\end{proof}

\begin{proposition}\label{pr}
	Let $R$ be a  ring, let $A$ be a finitely generated
	projective $R$-module, and let $J$ be an $R$-submodule of $A$.
	Suppose that there exists an injective $R$-linear map
$
	\varphi:A\hookrightarrow_R A/J.
	$
	If every injective endomorphism of $A$ is an automorphism, then
$J=0.
	$ \end{proposition}

\begin{proof}
	Let
$
	\pi:A\twoheadrightarrow A/J
	$	be the canonical surjection. Since $A$ is projective as an
	$R$-module, the map $\varphi:A\hookrightarrow_R A/J$ lifts along
	$\pi$. Thus there exists an $R$-linear map
$	g:A\to A$	
	such that
$\pi\circ g=\varphi$:
\[
\begin{tikzcd}
	&& A \arrow[dl, "\exists g"] \arrow[d, "\varphi"]   \\
	0 \longrightarrow J\arrow[r] & A \arrow[r, "\pi"'] & A/J \arrow[r] & 0
\end{tikzcd}
\]
	We claim that $g$ is injective. Indeed, if $g(a)=0$, then
$
	\varphi(a)=(\pi\circ g)(a)=0.
	$	Since $\varphi$ is injective, it follows that $a=0$. Hence $g$ is
	injective.
	By hypothesis, every injective endomorphism of $A$ is an
	automorphism. Therefore $g$ is an automorphism of $A$. 	Since $\varphi$ is injective, we have
	\(
	\ker(\varphi)=0.
	\)
	Since
	\[
	0=\ker(\varphi)
	=\ker(\pi\circ g)
	=g^{-1}(\ker\pi)
	=g^{-1}(J),
	\]
	we obtain
	\(
	g^{-1}(J)=0.
	\)
	Since $g$ is surjective, for every $j\in J$ there exists
	$a\in A$ such that $g(a)=j$. But then
	\(
	a\in g^{-1}(J)=0,
	\)
	and hence $j=g(a)=0$. Therefore
	\(
	J=0.
	\)
\end{proof}

\begin{corollary}
	\label{corh}
Let $R$ be zero-dimensional. Let $A$ be a finite $R$-algebra projective as module. Suppose  for some ideal $J$  of
$A$  there exists an $R$-algebra monomorphism
$A\hookrightarrow_R A/J.
$
Then
$J=0.$
\end{corollary}
\begin{proof}
	Since \(A\) is finite and projective over the zero-dimensional ring \(R\),
	every injective \(R\)-linear endomorphism of \(A\) is an automorphism.
	Hence Proposition~\ref{pr} applies.
\end{proof}

It is easy to construct rings with nilpotent elements satisfying Serre’s condition $(S_n)$.

\begin{proposition}\label{xpr}
	Let $(R,\fm)\to (A,\fm_A)$ be an extension of local rings so that $A$ is finitely generated
	extension and   $(S_2)$  as an $R$-module. Let $J\in \Spec(A)^\circ:=\Spec(A)\setminus\{\fm_A\}$.
	Suppose that there exists an injective $R$-linear map
	$
	\varphi:A\hookrightarrow_R A/J.
	$
	Then
	$J=0.
	$ \end{proposition}

\begin{proof}
	We proceed by induction on $d:=\dim(R)$ to observe $J=0$.
	First assume that $ d=0$. According to Lemma 
	\ref{thm:artinian-base}, $J=0$. We apply the inductive hypothesis, and observe that $J$ is zero over punctured spectrum of $R$.
	It turns out that $C:=\coker(\varphi)$ is of finite length.
	
	Here,  we separate the case $d=1$. This  implies that $J\in\min(A)$.
	Set $P=J\cap R$. It turns out that $
	\varphi_J:A_J\hookrightarrow_{R_P} A_J/JA_J.
	$ Recall that $\dim(R_P)=0$. Again, we apply the  case $d=0$, and deduce that $JA_J=0$. Consequently, $J=0$. 
	
	Then we may and do assume that $d>1$. By the independence theorem for local cohomology modules, we have $H^0_\fm(A/J)\cong_{}H^0_{\fm_A}(A/J)$, and this is zero because $J$ is a prime ideal and $J\neq \fm_A$, which imply that  $\depth(A/J)>0$. Similarly, $H^1_\fm(A)=0$ by  the
	$(S_2)$ condition applied along with $d>1$, and with the convenience that $\dim(R)=\dim(A)$, as any finite extension is integral (see \cite[9.1~(a)]{Matsumura}). 
	Consider the exact sequence
	$$0\longrightarrow A
	\xrightarrow{\ \varphi\ }
	A/J\longrightarrow C\longrightarrow0,$$and deduce the exact sequence
	$$0= H^0_{\mathfrak m }(A/J)\longrightarrow H^0_{\mathfrak m }(C)\longrightarrow H^1_{\mathfrak m }(A)=0.$$	Also, $H^0_\fm(C)\cong_{R}C$, as $C$ is of finite length as an $R$-module.
	Thus $C=0$. Here, we show $J=0$. Suppose on the way of contradiction that $J\neq 0$. Since $C=0$ we have $ 	A\cong_R A/J 
	$. This allows us to apply 	{Proposition} \ref{prop:hopfian} which say \(A\) admits a surjective \(R\)-algebra endomorphism which is not
	injective. Since $A$ is assumed to be finitely generated as an $R$-module, \cite[Theorem 2.4]{Matsumura} gives us a contradiction that we searched for it.  Then $J=0$.
\end{proof}

	Cohen-Macaulay rings are natural generalization of artinian ring. In this regard we present:
	
\begin{proposition}
	\label{thm:CM-reduction}
Let $(R,\mathfrak m)$ be a noetherian local  ring Let $A$ be a finite $R$-algebra so that there exists an
	$R$-algebra monomorphism
	$\varphi:A\hookrightarrow_RA/J  $ with $C:=\coker(\varphi)$.
If  
$C$ is maximal Cohen--Macaulay over $R$, then
$J=0.$ 
	\end{proposition}

	\begin{proof}
We argue by induction on $d:=\dim(R)$.
If $d=0$, Lemma \ref{thm:artinian-base} implies that $J=0$.
Assume now that $d>0$.
We may choose
a parameter ideal	$ \underline{x}:=x_1,\cdots, x_d\subseteq \mathfrak m $
which is $C$-regular.
	Consider the exact sequence
$$0\longrightarrow A
		\xrightarrow{\ \varphi\ }
		A/J\longrightarrow C\longrightarrow0\quad(\ast) $$ 	We obtain by tensoring
	$(\ast)$	with $R/(\underline{x})$, 
	the following exact sequence
 $$\operatorname{Tor}_1^R(R/(x),C)\longrightarrow A/(\underline{x})
\longrightarrow
 (A/J)\otimes_R  R/(\underline{x}) \longrightarrow C\otimes_R R/(\underline{x}) \longrightarrow0. $$ 
		But, $\operatorname{Tor}_1^R(R/(\underline{x}),C)=0
		$ because $\underline{x}$ is $C$-regular (see \cite[Ex. 1.1.12 (a)]{BH}).
	 	Thus
	 $$
		A/\underline{x}A\hookrightarrow (A/J)/\underline{x}(A/J)=(A/\underline{x})/(J+\underline{x})(A/\underline{x})
		$$ 	is an $R/(\underline{x})$-algebra monomorphism.
	 	By induction,
	 $
		J+\underline{x}A=\underline{x}A.
		$ 	Hence
 $
		J\subseteq \underline{x}A.
		$ Working with $\underline{x}^n:=x_1^n,\cdots, x_d ^n$ gives
  $
 J\subseteq \underline{x}^nA.
 $
 Consequently,   
   $$
 J\subseteq\bigcap_{n>0} \underline{x}^nA \subseteq\bigcap_{n>0} (\underline{x}A)^n.
 $$  So by Krull's intersection theorem, we deduce that
 $
		J=0.
		$ \end{proof}

The finiteness assumption in {Proposition}~\ref{thm:CM-reduction} is  needed, even if  $C$ is big Cohen-Macaulay:

\begin{example}\label{bas13}
Let    $R$ be  local ring, $x$ is a variable. Let \(M=R[x]u\oplus R[x]v\) be a free \(R[x]\)-module with basis \(\{u,v\}\). Consider the trivial extension \(A:=R[x]\ltimes M\) and let \(J:=Av\). Then \(A/J=R[x]\oplus R[x]u\) with \(u^2=0\). Define \(\varphi:A\to A/J\) to be the \(R\)-algebra homomorphism such that \(\varphi(x)=x^2\), \(\varphi(u)=u\), and \(\varphi(v)=xu\). By \cite[Example  1.3]{Bass} we see \(\varphi(A)= R[x^2]\oplus R[x]u \)
and \(\varphi(A)\) is a direct sum, hence \(\varphi\) is injective. Now, consider the exact sequence
$$0\longrightarrow A
\xrightarrow{\ \varphi\ }
A/J\longrightarrow C\longrightarrow0. $$ It follows that $C=R[x]/R[x^2]\cong_R\bigoplus Rx^{odd}$  is free as an $R$-module, and so  $C$ is a big Cohen-Macaulay over $R$. 
\end{example}

In opposite direction, we have a situation when $C=0$.

\begin{proposition}\label{op}
Let $(R,\fm )$ be local non  Cohen-Macaulay and  $A$ be an $R$-algebra finite and  over $R$ so that there is $\varphi:	A\hookrightarrow_R A/J$. If 	$C:=\coker(\varphi)$ is of finite projective dimension, then
$C=0.$ 	In particular, $J=0$.
\end{proposition}

\begin{proof}
		We proceed by induction on $\dim(R)$ to observe $C=0$.
First assume that $\operatorname{dim}(R)=0$. According to Lemma 
\ref{thm:artinian-base}, $C=J=0$. We apply the inductive hypothesis, and observe that $C$ is zero over punctured spectrum.
It turns out that $C$ is of finite length. Suppose on the way of contradiction that $C$  is nonzero. This implies that $R$ is Cohen-Macaulay (see \cite[Ex. 9.4.9]{BH}).
But this excluded by the assumption. So, $C=0$. 
Here, we show $J=0$. Suppose on the way of contradiction that $J\neq 0$. Since $C=0$ we have $ 	A\cong_R A/J 
$. This allows us to apply 	{Proposition} \ref{prop:hopfian} which say \(A\) admits a surjective \(R\)-algebra endomorphism which is not
injective. Since $A$ is assumed to be finitely generated as an $R$-module, \cite[Theorem 2.4]{Matsumura} gives us a contradiction that we searched for it.  Then $J=0$.
\end{proof}

\begin{remark}
\begin{enumerate}
\item[(i)]  The finiteness assumption in {Proposition}~\ref{op} is  needed, it is enough to start 
{Example} \ref{bas13} with any non   Cohen-Macaulay ring $R$. Recall that we observed in 	{Example} \ref{bas13} that $C$  were be free, and so be of finite projective dimension.
	
\item[(ii)] Recall that the existence of a nonzero  $\fm$-torsion module of finite projective dimension implies the ring is Cohen-Macaulay. This turns out that the   finiteness assumption in
the artinan situation (see Lemma 
\ref{thm:artinian-base}) is needed too, as otherwise {Proposition}~\ref{op} combined with part (i) leads a contradiction.
\end{enumerate}

\end{remark}

	\maketitle

\section{Question 1.1 is true}

In this section we show:

\begin{theorem}\label{thm:answer}
	Let $R$ be a commutative ring, let $A$ be a reduced finitely generated
	$R$-algebra, and let $J$ be an ideal of $A$. If there is an injective
	$R$-algebra homomorphism
	\(
	\varphi:A\hookrightarrow_R  A/J,
	\)
	then $J=0$.
\end{theorem}

\begin{proof}
	Let $\iota:R\to A$ be the structure map and let
	$\pi:A\to A/J$ be the quotient map. Choose $a_1,\ldots,a_n\in A$
	such that $A=\iota(R)[a_1,\ldots,a_n]$. For each $i$, choose
	$f_i\in R[X_1,\ldots,X_n]$ with
	\[
	\varphi(a_i)=\pi\bigl(f_i(a_1,\ldots,a_n)\bigr).
	\]
	Evaluation of polynomials over $R$ uses the structure map $\iota$.
	
	Fix $j\in J$. Choose $g\in R[X_1,\ldots,X_n]$ such that
	$j=g(a_1,\ldots,a_n)$. Let $c_1,\ldots,c_m\in R$ comprise all coefficients
	occurring in $f_1,\ldots,f_n,g$, and let $R_0\subseteq R$ be the subring
	generated by $1_R,c_1,\ldots,c_m$. There is a surjective ring homomorphism
	\[
	\mathbb Z[T_1,\ldots,T_m]\longrightarrow R_0,
	\qquad T_\ell\longmapsto c_\ell.
	\]
	Thus $R_0$ is noetherian.
	Set
	\[
	B=\iota(R_0)[a_1,\ldots,a_n]\subseteq A,
	\qquad L=B\cap J.
	\]
	The ring $B$ is reduced because it is a subring of $A$. It is a finitely
	generated $R_0$-algebra. The kernel of the natural surjection
	$R_0[X_1,\ldots,X_n]\to B$ is finitely generated, since its
	source is noetherian. Hence $B$ is finitely presented over $R_0$.
	The coefficients of $g$ belong to $R_0$, so $j\in L$.
	
	For $r\in R_0$, $R$-algebra compatibility gives
	$\varphi(\iota(r))=\pi(\iota(r))$. The coefficients of the $f_i$ also belong to
	$R_0$. It follows that
	\[
	\varphi(B)\subseteq \pi(B).
	\]
	The restriction $\pi|_B$ has kernel $L$ and induces an $R_0$-algebra isomorphism
	\[
	\theta:B/L\xrightarrow{\ \sim\ }\pi(B),
	\qquad b+L\longmapsto q(b).
	\]
	Consequently
	\[
	\psi=\theta^{-1}\circ\varphi|_B:B\longrightarrow B/L
	\]
	is an injective $R_0$-algebra homomorphism. Indeed, $\varphi|_B$ is injective
	and $\theta$ is an isomorphism. Bass's theorem (see  {Theorem} \ref{bt}), applied to $B$ and $L$, yields
	$L=0$. Therefore $j=0\in B\subseteq A$. Since $j\in J$ was arbitrary, $J=0$.
\end{proof}

\begin{remark}[Characteristic]
	The homomorphism $\mathbb Z\to R_0$ sends $k$ to
	$k\cdot1_{R_0}$ and need not be injective. In characteristic $n>0$, the
	polynomial-ring surjection above factors through
	$(\mathbb Z/n\mathbb Z)[T_1,\ldots,T_m]$; for prime $n=p$, this is
	$\mathbb F_p[T_1,\ldots,T_m]$. Thus the proof applies in every
	characteristic. The structure map $R_0\to B$ may also have a kernel.
\end{remark}

\section{Toward module version}
 The preceding results suggest that the essential obstruction to embedding a module
 into a proper quotient is controlled by the zero-divisors of the base ring. We
 begin with the following observation.

\begin{proposition}\label{31}
	Let \(R\) be a noetherian reduced ring, \(J \triangleleft R\) an ideal, and \(M\) a finitely generated torsion-free \(R\)-module.
	If \(M \hookrightarrow_R (M/JM)\), then \(J \subseteq \mathfrak{p}\) for some \(\mathfrak{p} \in \operatorname{Min}(R)\). Assume in addition  \(M\) is
	faithful. Then $J=0$.
\end{proposition}
\begin{proof}
	Let
	$
	\varphi:M\longrightarrow M/JM
	$	be an $R$-linear injection. For every $j\in J$ and $m\in M$, we have
	$
	jm\in JM,
	$ and hence
	$
	j\varphi(m)=\varphi(jm)=0
	$ in $M/JM$. Since $\varphi$ is injective, it follows that
	$
	jm=0
	$	for every $m\in M$. Therefore
	$
	JM=0.
	$	We claim that every element of $J$ is a zero-divisor of $R$. Indeed, suppose that
	$j\in J$ is regular. Since $M$ is torsion-free, multiplication by $j$ on $M$ is injective. But $jM=0$, and $M\neq0$, a contradiction. Hence
	$
	J\subseteq \operatorname{Zd}(R).
	$ Since $R$ is noetherian and reduced,
	$\operatorname{Zd}(R)
	=
	\bigcup_{\mathfrak p\in\operatorname{Min}(R)}\mathfrak p.
	$ 	Moreover, $\operatorname{Min}(R)$ is finite. Consequently,
	$$
	J\subseteq
	\bigcup_{\mathfrak p\in\operatorname{Ass}(R)}\mathfrak p=\bigcup_{\mathfrak p\in\operatorname{Min}(R)}\mathfrak p.
	$$ By prime avoidance, there exists
	$
	\mathfrak p\in\operatorname{Min}(R)
	$
	such that
	$
	J\subseteq\mathfrak p.
	$ 
	
	Assume now  \(M\) is
	faithful. Then
	\[
	[M\lhook\joinrel\longrightarrow_R M/JM
	\qquad\Longrightarrow\qquad
	J=0].
	\]
The first paragraph gives \(JM=0\). Since \(M\) is faithful,
	\(\Ann_R(M)=0\), and hence \(J=0\).	
\end{proof}

\begin{remark}
	The faithfulness assumption is important since \(J := \operatorname{ann}(M)\). Then \(M \cong M/ JM = M/0 = M\).
\end{remark}
	
The preceding proposition uses reducedness only through the description of the
zero-divisors as the union of the minimal primes. We next give a different
argument which, in dimension one, does not require \(A\) to be reduced.
	\begin{proposition}\label{lc}
		Let $R\to A$ be local and finite morphism of noetherian local rings where $\dim(A)=1$ so that there exists an
		$R$-algebra monomorphism
		$\varphi:A\hookrightarrow_R A/J  $. Then \(J \subseteq \mathfrak{p}\) for some \(\mathfrak{p} \in \operatorname{Min}(A)\).	
		
	\end{proposition}

	\begin{proof}
		Suppose not. By $\Ht(J)>0$, and so $J$ is $\mathfrak m_A$-primary. Then
		by Grothendieck's vanishing theorem, $H^1_{\mathfrak m_A}(A/J)=0$.
	It is easy to see $\rad(\mathfrak m_RA)=\mathfrak m_A$. 	By the independent theorem for local cohomology modules,
		$H^i_{\mathfrak m_A}(-)\cong_RH^i_{\mathfrak m_R}(-)$. In particular,
		$H^1_{\mathfrak m_R}(A/J)=0$. Now, we use the natural short exact sequence of $R$-modules $0 \to A\to A/J\to  C\to 0,$ to deduce the exact sequence
		$$ H^0_{\mathfrak m_R}(C)\longrightarrow H^1_{\mathfrak m_R}(A)\longrightarrow H^1_{\mathfrak m_R}(A/J)=0.$$Now, recall from finiteness of $R\to A$ that  $A/J$ and so $C$ is finitely generated over $R$. Consequently, $ H^0_{\mathfrak m_R}(C)$ is finitely generated. Then $H^1_{\mathfrak m_R}(A)$ is finitely generated as an $R$-module.
		Therefore, $H^{\dim(A)}_{\mathfrak m_A}(A)$ is finitely generated as an $A$-module.
		This contradiction shows that \(J \subseteq \mathfrak{p}\) for some \(\mathfrak{p} \in \operatorname{Min}(A)\).
	\end{proof}
 
\begin{question}
Is in Proposition  \ref{31} the noetherian assumption important?
\end{question}
	
The answer is affirmative.

\begin{example}\label{E3.6}
	Let \(k\) be a field and set
$	R=
	k[X_1,X_2,\ldots]\big/\bigl(X_iX_j\mid i\neq j\bigr).$
For each \(n\geq1\), let
$
	\mathfrak p_n=(x_i\mid i\neq n).
$
	Then
$
	R/\mathfrak p_n\cong k[x_n],
$
	so each \(\mathfrak p_n\) is prime. Moreover,
\(R\) is reduced and
$
	\Min(R)=\{\mathfrak p_n\mid n\geq1\}.
$ Let
	\[
	J=(x_1,x_2,\ldots)
	\qquad\text{and}\qquad
	M=R/J.
	\]
	Since \(JM=0\), the canonical map
$
	M\to M/JM
$
	is an isomorphism. We claim that \(M\) is torsion-free.
Every element of \(R\) can be written in the form
\[
	r=c+f_1(x_1)+\cdots+f_t(x_t),
	\]
	where \(c\in k\) and \(f_i(x_i)\in(x_i)\). If \(c=0\), then, for any
	\(n>t\),
$
	r\in\mathfrak p_n,
$	and hence \(r\) is a zero-divisor. Consequently, every regular element of \(R\)
	has nonzero constant term.
Now
$M=R/J\cong k.$
	If \(r\in R\) is regular, its action on \(M\) is multiplication by its nonzero
	constant term. Hence multiplication by \(r\) on \(M\) is injective. Thus \(M\)
	is torsion-free.
	However, for every \(n\geq1\),
	\(
	x_n\in J	\)
 {but} 
	\(	x_n\notin\mathfrak p_n.
	\)
	Therefore
$
	J\not\subseteq\mathfrak p_n
$
	for every $(n\geq1)$. Since the \(\mathfrak p_n\) are precisely the minimal
	primes of \(R\), we obtain
$J\not\subseteq\mathfrak p$
 for every $\mathfrak p\in\Min(R).$ 
	Thus
$
	M\lhook\joinrel\rightarrow_R M/JM
	$
	does not imply the conclusion of Proposition~\ref{31} when \(R\) is not
	noetherian.
\end{example}
		But, we can reduce the  finitely generated of $M$ to only has rank (see \cite[\S 1.4]{BH}).

\begin{corollary}
Let $R$ be a noetherian reduced ring, let $J\triangleleft R$ be an ideal, and let $M\neq 0$ be a torsion-free $R$-module and has rank. If
		$
		M\hookrightarrow_R M/JM,
		$ then
		$
		J\subseteq\mathfrak p
		$ 	for some $\mathfrak p\in\operatorname{Min}(R)$.
	\end{corollary}
	
	\begin{proof}
		The proof of Proposition~\ref{31} applies verbatim. Indeed, the only
		properties of \(M\) used there are \(M\neq0\) and the fact that every regular
		element of \(R\) acts injectively on \(M\). No finite generation assumption is
		required.
	\end{proof}
	
	\begin{observation}\label{inj}
		Let \(R\) be a noetherian reduced ring, \(J\subseteq R\) an ideal, and \(M\) a nonzero divisible \(R\)-module such that \(JM=0\). Then \(J\) is contained in some minimal prime ideal of \(R\).
	\end{observation}
	
	\begin{proof}
		If by divisible we mean divisible by every regular element of \(R\), then the argument is very clean.
		Suppose \(0\neq M\) is divisible and \(JM=0\). If \(j\in J\) were regular, divisibility gives
		\(
		M=jM.
		\)
		But \(JM=0\) gives \(jM=0\). Hence \(M=0\), contradiction.
		Thus every \(j\in J\) is a zero-divisor. Since \(R\) is noetherian and reduced,
		\(
		\operatorname{Zd}(R)=\bigcup_{\mathfrak p\in\operatorname{Min}(R)}\mathfrak p.
		\)
		Because there are finitely many minimal primes, prime avoidance gives
		\(
		J\subseteq\mathfrak p
		\)
		for some \(\mathfrak p\in\operatorname{Min}(R)\).
	\end{proof}

We continue by extending {Proposition} \ref{31}.  To this end, we record a remark.

\begin{remark}\label{rem:ass-general}
	For a general module \(M\), the set \(\operatorname{Ass}_R(M)\) may be empty or
	infinite. If \(R\) is Noetherian and \(M\neq 0\) is finitely generated,
	then \(\operatorname{Ass}_R(M)\) is finite and nonempty. In the non-Noetherian setting,
	one sometimes uses instead the \emph{weakly associated primes}:
	\(\mathfrak p\in\operatorname{wAss}_R(M)\) if \(\mathfrak p\) is a prime ideal
	minimal over \(\operatorname{Ann}_R(m)\) for some \(m\in M\). For more details, see \cite[3.3.1]{G}.
\end{remark}
	
\begin{theorem}\label{41}
	Let \(R\subseteq A\) be an extension of noetherian reduced rings such
	that \(A\) is torsion-free as an \(R\)-module. Let \(J\triangleleft A\),
	and let \(M\) be a finitely generated \(R\)-module which is torsion-free
	as an \(A\)-module. If
	\(
	M\hookrightarrow_R M/JM,
	\)
	then
	\(
	J\subseteq\mathfrak p
	\)
	for some \(\mathfrak p\in\operatorname{Min}(A).
	\)
\end{theorem}
	
	\begin{proof}
	First, \(R\subseteq A\) and \(M\) is torsion-free  as an \(A\)-module imply that
 \(M\) is torsion-free  as an \(R\)-module.
 	By \cite[Page 183]{Matsumura}, a ring is reduced if and only if it satisfies \((R_0)\) and \((S_1)\).
		
		\begin{align*}
			(R_0): & \quad R_\mathfrak{p} \text{ is regular } \forall \mathfrak{p} \text{ with } \operatorname{ht}(\mathfrak{p}) = 0 \Rightarrow R_\mathfrak{p} \text{ is a field}. \\
			(S_1): & \quad \forall \mathfrak{p} \in \operatorname{Ass}(R) \Rightarrow \mathfrak{p} \in \operatorname{Min}(R).
		\end{align*}
		
		Now, let \(\mathfrak{p} \in \operatorname{Ass}(R)\). Consider the exact sequence
		\(
		0 \to JM \to M \to M/JM \to 0.
		\)
		Localizing at \(\mathfrak{p}\), we obtain
		\[
		0 \to 	(JM)_\mathfrak{p} \to M_\mathfrak{p} \to (M/JM)_\mathfrak{p} \to 0.
		\]
		Here, $\operatorname{vdim}(-) $ means the vector space dimension over the   field $R_\mathfrak{p}$.	On the one hand we have the additivity property of dimension on short exact sequence. Namely, we have
		\[
		\operatorname{vdim}(M_\mathfrak{p}) = \operatorname{vdim}((M/JM)_\mathfrak{p})+\operatorname{vdim}((JM)_\mathfrak{p})\ge \operatorname{vdim}((M/JM)_\mathfrak{p}).
		\]
		On the other have and in view of the injection   \(M_\mathfrak{p} \to (M/JM)_\mathfrak{p}\), we have \[
		\operatorname{vdim}(M_\mathfrak{p}) \leq \operatorname{vdim}((M/JM)_\mathfrak{p}).
		\]Combining these yield  \(
		\operatorname{vdim}(M_\mathfrak{p}) =\operatorname{vdim}((M/JM)_\mathfrak{p}).
		\)
		Consequently, 	
		\begin{align*}
			\operatorname{vdim}(JM)_\mathfrak{p} = 0  
			\Rightarrow (JM)_\mathfrak{p} = 0\quad(\ast)
		\end{align*}

	Since $\operatorname{wAss}$ is subset of support, $\mathfrak{p} \notin\operatorname{wAss}(JM)$.	But \(JM \subseteq M\), so
		\[
	\operatorname{wAss}_R(JM) \subseteq \operatorname{wAss}_R(M) \subseteq \operatorname{wAss}_R(R)=\operatorname{Ass}_R(R).
		\]
		It turns out from  $(\ast)$ that  \(\operatorname{Ass}_R(JM) = \emptyset\) and so \(JM = 0\), because it is torsion-free over $R$. In particular, it is zero as an $A$-module which gives
		\[
		J  \subseteq \operatorname{Zd}_A(M) =\bigcup_{\mathfrak{p} \in \operatorname{wAss}(M)} \mathfrak{p} \subseteq  \bigcup_{\mathfrak{p} \in \operatorname{Ass}(A)} \mathfrak{p}.
		\]
		Therefore, by prime avoidance \(J \subseteq \mathfrak{p}\) for some \(\mathfrak{p} \in \operatorname{Ass}(A)\).
	\end{proof}
	
\begin{corollary}\label{42}
	Adopt the above notation.	If \(M\) is in addition faithful over $A$, then \(J = 0\).
\end{corollary}

	But, we can reduce the  finitely generated of $M$ to only has rank (see \cite[\S 1.4]{BH}).

\begin{remark}
	Let \(R\subseteq A\) be a noetherian reduced ring, \(J \triangleleft A\) an ideal, and \(M\) has rank over $R$ and torsion-free  as an \(A\)-module.
If \(M \hookrightarrow_R (M/JM)\), then \(J \subseteq \mathfrak{p}\) for some \(\mathfrak{p} \in \operatorname{Min}(A)\).
\end{remark}
The finiteness of $M$ as $R$-module is important:
\begin{example}
	Let $R:=k$ be a field, set
$
M=A=k[x,y]$ and $ J=(x).$	Then $A$ is a Noetherian reduced ring and $M$ is a finitely generated torsion-free
	$A$-module.
 We have
 $
	M/JM
	=
	A/xA
	\cong k[y].
	$ 	
	As $k$-vector spaces,
 $
	A=R[x,y]$ {and}\ 
	$R[y]
	$ 	both have countably infinite dimension over $k$. Hence there exists a
	$k$-linear isomorphism
 $
	\varphi:A\to A/xA.
	$ 	For example, choose a bijection
$
	\pi:\mathbb N^2\to\mathbb N
	$ 	and define
 $
	\varphi(x^a y^b)=y^{\pi(a,b)}$
 with $(a,b\geq0),$ 	extending $k$-linearly.
	Thus
$
	M\hookrightarrow_R M/JM.
	$
	However,
 $
	\operatorname{Min}(R)=\{(0)\},
$
	whereas
$
	J=(x)\neq(0).
	$ Consequently,
 $
	J\not\subseteq\mathfrak p$
 {for every } $\mathfrak p\in\operatorname{Min}(R).$ 
\end{example}

	\maketitle
\section{The dual surjection problem}

We now consider a dual version of Question~ \ref{Bb}. Instead of asking whether
\(A\) can be embedded into a proper quotient \(A/J\), we ask whether a
nonzero ideal \(J\) can be mapped onto \(A\), or whether \(A\) can be mapped
onto \(J\).

\begin{question}\label{dBb} 
	Let \(A\) be a finitely generated \(R\)-algebra and let \(0\neq J\subseteq A\)
	be an ideal. (i): Suppose that there exists an \(R\)-linear epimorphism
	\(
	A\twoheadrightarrow_R J,
	\)
	or (ii):  suppose there  is an \(R\)-linear epimorphism
	\(J\twoheadrightarrow_R A.
	\)
	Under what hypotheses must \(J\cong_R A\)?
\end{question}

	 {Question} \ref{dBb}(i) is not  true even $\pi$ is splitting, for example let $A=\oplus_{n\geq 0}A_n$ be graded ring and let
	 $J=\oplus_{n> 0}A_n$ be the irrelevant ideal and $R:=A_0$.
	 The explicit finite extension is given by Remark \ref{53}.
	 
 \begin{remark}\label{rem:dual-special}
Suppose \(R=A\). If there exists an \(A\)-linear epimorphism
$A\twoheadrightarrow J,	$ then \(J\) is a cyclic \(A\)-module, and hence \(J\) is a principal ideal. On the other hand, if there exists an \(A\)-linear epimorphism
 the epimorphism splits.
Thus
 $ 	J\cong A\oplus\ker(\pi).$ In particular, if \(A\) is an integral domain, then
$$\operatorname{rank}_A(J)=1=\operatorname{rank}_A(A),$$
and consequently \(\ker(\pi)=0\). Hence \(J\cong A\).
	 \end{remark}

	\begin{theorem}\label{thm:dual-domain}
	Let $R\subseteq A$ be integral domains such that $A$ is finite as an
	$R$-module. Let $0\neq J\subseteq A$ be an ideal. Then
	$
	\operatorname{rank}_R(J)=\operatorname{rank}_R(A).
	$ 	Moreover, if there exists either an $R$-linear epimorphism
	$
	\pi:A\twoheadrightarrow_R J
	$ or an $R$-linear epimorphism
	$
	\pi:J\twoheadrightarrow_R A,
	$ then $\pi$ is an isomorphism. In particular,
	$
	A\cong_R J.
	$ 
\end{theorem}

\begin{proof}
	Let $Q:=\operatorname{Frac}(R)$. Since $A$ is a domain and is finite
	over $R$, the $Q$-algebra
	$A_Q:=A\otimes_R Q
	$
	is a finite-dimensional domain over the field $Q$. Hence $A_Q$ is a
	field.
	Since $0\neq J\subseteq A$, we have
	$$
	0\neq J\otimes_R Q\subseteq A\otimes_R Q=A_Q.
	$$
	As $A_Q$ is a field, it follows that
	$
	J\otimes_R Q=A_Q.
	$	Consequently,
	
	$$
	\operatorname{rank}_R(J)
	=\operatorname{vdim}_Q(J\otimes_R Q)
	=\operatorname{vdim}_Q(A\otimes_R Q)
	=\operatorname{rank}_R(A).
	$$
	
	Suppose first that there exists an $R$-linear epimorphism
	$\pi:A\twoheadrightarrow_R J.$	
	Let $K:=\ker(\pi)$. Since $Q$ is flat over $R$, tensoring the exact
	sequence
$0\to K\to A\xrightarrow{\pi}J
\to0$
	with $Q$ gives an exact sequence
	$$0\longrightarrow K\otimes_RQ
	\longrightarrow A\otimes_RQ
	\longrightarrow J\otimes_RQ
	\longrightarrow0.
	$$
	Since
	$A\otimes_RQ=J\otimes_RQ,$  we obtain
	$K\otimes_RQ=0. $ But $K\subseteq A$ and $A$ is torsion-free as an $R$-module. Hence
	$K$ is torsion-free as an $R$-module, and therefore
	$
	K=0.
	$	Thus $\pi$ is an isomorphism.
	
	Now suppose that there exists an $R$-linear epimorphism
$\pi:J\twoheadrightarrow_R A.
	$ Again let $K:=\ker(\pi)$. Tensoring
$0\to K\to J\xrightarrow{\pi}A
	\to0$	with $Q$ gives
	
	$$
	0\longrightarrow K\otimes_RQ
	\longrightarrow J\otimes_RQ
	\longrightarrow A\otimes_RQ
	\longrightarrow0.
	$$
	
	Since
	$
	J\otimes_RQ=A\otimes_RQ,
	$ 	we obtain
	$
	K\otimes_RQ=0.
	$	As $K\subseteq J\subseteq A$ and $A$ is torsion-free over $R$, we again
	obtain $K=0$. Hence $\pi$ is an isomorphism.
	Therefore, in either case,
	$
	A\cong_R J.
	$	\end{proof}
 
 \begin{remark}\label{53}
 	The domain hypothesis in Theorem~\ref{thm:dual-domain} cannot be omitted
 	in general. Let
  $ A=k[x,y]/(x,y)^2$ and $J=(x,y).$ 	Then
  $ 	A=k\oplus kx\oplus ky,$ and 
 	$J=kx\oplus ky. $ 	Define a \(k\)-linear epimorphism
 $ 	A\twoheadrightarrow _kJ,
 	$ by assignments
  $
 	1\longmapsto x$, $x\longmapsto y$ and $y\longmapsto0.$ This map even admits a \(k\)-linear section. Nevertheless,
 $ J\neq A. $ 	Thus an epimorphism \(A\twoheadrightarrow_R J\) need not imply
 	\(J\cong_R A\) over an artinian base.
 \end{remark}

 \begin{example}\label{ex:artinian-dual}
 	Let
 $
 	A=R=k[x,y]/(x,y)^2
 	$ 	and let \(J=(x)\). Define an \(R\)-linear map
  $
 	\pi:A\to J
 	$	by
 $
 	\pi(1)=x$,  $\pi(x)=0$ and  $\pi(y)=0.$
 Then \(\pi\) is surjective. However,
  $
 	J\neq A. $
  Thus Question~\ref{dBb}(i) has a negative answer even when
 	\(R=A\) is artinian.
 \end{example}

 	{Question} \ref{dBb}(ii) is true, if $R$ is artinian:

 \begin{proposition}\label{prop:artinian-dual}
 	Let \(R\) be an artinian ring, let \(A\) be a finite \(R\)-algebra, and let
 	\(0\neq J\subseteq A\) be an ideal. If there exists an \(R\)-linear
 	epimorphism
 $	J\twoheadrightarrow_R A,$	then this epimorphism is an isomorphism. In particular,
 $J\cong_R A.$
 \end{proposition}
 
 \begin{proof}
 	Since \(J\subseteq A\), we have
 $ \ell_R(J)\leq\ell_R(A). 	$ On the other hand, an \(R\)-linear epimorphism \(J\twoheadrightarrow A\)
 	gives
 $ \ell_R(J)\geq\ell_R(A). $	Hence
 $ \ell_R(J)=\ell_R(A). $	If \(K\) denotes the kernel of the epimorphism, then
 $
 	0\to K\to J\to A\to0
 	$ gives
 $\ell_R(J)=\ell_R(K)+\ell_R(A).$ Thus \(\ell_R(K)=0\), and consequently \(K=0\).
 \end{proof}

\begin{proposition}\label{mcmno}
	Let \((R,\mathfrak m)\) be a noetherian local ring of positive dimension,
	let \(R\to A\) be a finite local homomorphism, and let \(J\) be an ideal
	of \(A\). Assume that \(A/J\) is a maximal Cohen--Macaulay
	\(R\)-module. Then there is no \(R\)-linear epimorphism
 $
	J\twoheadrightarrow_R A.
	$ 	
\end{proposition}

\begin{proof}
	Suppose, to the contrary, that there exists an \(R\)-linear epimorphism
 $
	\pi:J\twoheadrightarrow A.
	$ 	Let \(d=\dim R\), and choose an \(R\)-regular system of parameters
 $
	\mathbf{x}=x_1,\ldots,x_d
	$ 	which is also \(A/J\)-regular. Tensoring the exact sequence
 $
	0\to J\to A\to A/J
	\to0
 $
 with \(R/(\mathbf{x})\), and using the \(A/J\)-regularity of
	\(\mathbf{x}\), gives an injection
 $
	J/\mathbf{x}J\to A/\mathbf{x}A 
	$ (see the argument presented in Proposition
	\ref{thm:CM-reduction}).	On the other hand, tensoring the epimorphism
 $ J\twoheadrightarrow _RA
	$ 	with \(\overline{R}:=R/(\mathbf{x})\) gives an epimorphism
 $
	J/\mathbf{x}J\twoheadrightarrow_{\overline{R}} A/\mathbf{x}A.
	$ 	Hence
	$$
\ell(	J/\mathbf{x}J)=\ell( A/\mathbf{x}A)\quad(\ast)
	$$	
Let us apply $(\ast)$ along with the following induced short  exact sequence
$$
	0\longrightarrow J/\mathbf{x}J
	\longrightarrow A/\mathbf{x}A
	\longrightarrow (A/J)/\mathbf{x}(A/J)
	\longrightarrow0,
	$$
to deduce that
$ (A/J)/\mathbf{x}(A/J)=0. $ 	By Nakayama's lemma,
 $A/J=0, $ 	contrary to the assumption that \(A/J\) is a nonzero maximal
	Cohen--Macaulay \(R\)-module.
\end{proof}
\begin{remark}
Let us start 
{Example} \ref{bas13} with a  Cohen-Macaulay ring $R$.	Recall from {Example} \ref{bas13} that $A/J$ is $R$-free and so big Cohen-Macaulay. In particular, the finite assumption of   {Proposition} \ref{mcmno} is  needed, even if $A$ is finitely generated as an $R$-algebra.
\end{remark}

\begin{corollary}\label{cormcmno}
	Suppose  $R\to A$ is an extension of local rings of  dimension one, and $P\in \operatorname{Ass}(A)\setminus\{\fm_A\}$.  Then there is no \( \pi:	P\twoheadrightarrow_R  A
\).  
\end{corollary}

\begin{proof}
Since 	 $P\in \operatorname{Ass}(A)\setminus\{\fm_A\}$ we see $0<\depth(A/P)$. We know $1\leq \depth(A/P)\leq\dim(A/P)\leq\dim(A)=1$, i.e., \(A/P\) is  maximal Cohen--Macaulay. It remains to apply {Proposition} \ref{mcmno}.	
\end{proof}

\begin{corollary}\label{ccormcmno}
	Suppose  $R=A$ is   a local ring  of  dimension one, and $P\in \operatorname{Ass}(A)$.  Then there is no \( \pi:	P\twoheadrightarrow_R  A
	\).  
\end{corollary}

\begin{proof}
By {Corollary} \ref{cormcmno},  it is enough to deal with $P=\fm \in \operatorname{Ass}(A)$. Suppose on the way of contradiction that there is 	
\(\pi:P\twoheadrightarrow_{R}  A\). Let $x$ be so that $P=(0:_Rx)$,
such a thing exists by definition of associated primes. Then $0=\pi(0)=\pi(xP)=x\pi(P)=xR$.
Thus $x=0$, and so $P=(0:_Rx)=(0:_R0)=R$, a contradiction, as primes are proper ideals.
\end{proof}

The assumption $P\in \operatorname{Ass}(A)$ in {Corollary} \ref{ccormcmno} is needed:
\begin{example}
	Let $A:=R:=k[[X]]$, and set $P:=(X)$.  There is a surjective  \(R\)-linear map \(\pi:P\twoheadrightarrow_{R}  A\) defined by 
	\(
	Xf(x)\mapsto f(x).
	\)
	\end{example}
Most of the results extend to the module setting. Here, we present a sample parallel to {Proposition} \ref{pr}.
	
\begin{observation}\label{inj2}
 Let $ M$ be an injective module over $A$.  
Let \(J\subseteq A\) be an ideal. If there exists an  
epimorphism
		$\pi:	JM\twoheadrightarrow_A M$	then  
		$JM=M.$ In particular, if $J$ is nilpotent, then $M=0$.
\end{observation}
	
\begin{proof}
Since  $M$ is injective, the map $\pi:	JM\twoheadrightarrow M$ extends to an $A$-homomorphism $\tilde{\pi}: M\twoheadrightarrow M$ which is epic as well:
\[\xymatrix{& 0 \ar[r] & JM \ar[r]^{\subseteq} \ar[d]_{\pi} & M \ar[dl]^{\exists \tilde{\pi}} \\
&& M}
\]

Since $\tilde{\pi}$ is $A$-linear, it follows that  $M=\pi(JM)=\tilde{\pi}(JM)\subseteq J\tilde{\pi}(M) =J  M .$ 
	By reaping this process, $M= J  M=J(JM)=\cdots =J^nM  $  for all $n$, which imply the last claim.
	\end{proof}

\end{document}